\documentclass[11pt]{amsart}
\usepackage{amssymb,amsmath, amsthm}
\usepackage{a4wide}
\usepackage{hyperref}
\usepackage{enumerate}
\usepackage{color}
\usepackage{textcomp}
\usepackage{graphicx}

\author{Ulf K\"uhn, Jan Steffen M\"uller}
\title[Lower bounds on $\om^2$]{Lower bounds on the arithmetic self-intersection number of the relative
dualizing sheaf on arithmetic surfaces}
\thanks{The second author was supported by DFG-grant KU
2359/2-1}
\newtheorem{thm}{Theorem}[section]
\newtheorem{prop}[thm]{Proposition}
\newtheorem{lemma}[thm]{Lemma}
\newtheorem{cor}[thm]{Corollary}

\theoremstyle{remark}
\newtheorem{rk}[thm]{Remark}
\newtheorem{defn}[thm]{Definition}
\newtheorem{ex}[thm]{Example}
\newcommand\Q{\mathbb{Q}}

\newcommand\sF{\mathcal{F}}
\newcommand\Z{\mathbb{Z}}
\newcommand\N{\mathbb{N}}
\newcommand\R{\mathbb{R}}

\newcommand\X{\mathcal{X}}

\renewcommand\L{\mathcal{L}}
\newcommand\M{\mathcal{M}}
\newcommand\K{\mathcal{K}}
\newcommand{\om}{\overline{\omega}}
\newcommand\Spec{\mathop{\rm Spec}\nolimits}

\renewcommand\O{\mathcal{O}}
\newcommand\Oa{\overline{\O}}

\newcommand{\ldot}{\,.\,}

\newcommand{\Div}{\operatorname{Div}}
\newcommand{\supp}{\operatorname{supp}}

\newcommand{\NT}{\operatorname{NT}}

\newcommand{\fp}{\mathfrak{p}}
\newcommand{\eps}{\varepsilon}

\begin{document}
 
\date{\today}
\begin{abstract}We give an explicitly computable lower bound for the arithmetic self-intersection
  number $\om^2$ of the dualizing sheaf on a large class  of arithmetic
surfaces. If some technical conditions are satisfied, then this lower bound is positive. 
In particular, these technical conditions are always satisfied for minimal arithmetic
surfaces with simple multiplicities and at least one reducible fiber, but we have also
used our techniques to obtain lower bounds for some arithmetic surfaces with non-reduced
fibers.
\end{abstract}
\maketitle
\section{Introduction}\label{s1}
Let $K$ be a number field, let $\O_K$ denote the ring of integers of $K$.
Let $X/K$ denote a smooth, projective, geometrically
irreducible curve of genus $g>1$,
let $\pi:\X\to \Spec(\O_K)$ be a proper regular model of $X$ and
let $\om=\om_\X$ denote the relative dualizing sheaf on $\X$ over $\Spec(\O_K)$, equipped
with the Arakelov metric. 
The arithmetic self-intersection $\om^2$ is one of the most fundamental objects
in arithmetic intersection theory; see for instance~\cite{Szpiro} for a discussion.
In this note we show how to effectively compute lower bounds on $\om^2$ in many
situations including, but not limited to, semistable $\X$.
To each $\Q$-divisor $D \in \Div_\Q(X)=\Div(X)\otimes_\Z\Q$ of degree one we attach in Definiton~\ref{LDDef} a
hermitian line bundle $\overline{\L_D}$.
 We show that the height $h_{\overline{\L_D}}(\, \cdot \,)$ with respect to
$\overline{\L_D}$ is closely related to the
N\'eron-Tate height induced by the embedding $j_D$ of $X$ into its Jacobian via $D$.
More precisely, we define a certain finite set $T(\X)$ of closed points on $\X$ and prove
the following result in Section~\ref{hi}, where we write $D_\X$ for the Zariski closure in
$\X$ of an irreducible divisor $D \in \Div(X)$ and extend this to $\Div_\Q(X)$ by
linearity.
\begin{thm}\label{HtNonNeg}
  Suppose that $D\in\Div_\Q(X)$ has degree one and $\supp(D_\X)\,\cap\,
  T(\X)=\emptyset$. 
  Then, if $E=\sum^d_{j=1} (P_j)$ is an irreducible
  divisor on $X$, where $P_j \in
  X(\bar{K})$, and $\supp(E_\X)\,\cap\, T(\X) = \emptyset$, we have 
  \[
    h_{\overline{\L_D}}(E) = \frac{1}{d} \sum^d_{j=1} h_{\NT}(j_D(P_j))\ge 0.
  \]
  In particular, we have $h_{\overline{\L_D}}(P)=h_\mathrm{NT}(j_D(P))$ for all $P\in X(K)$.
\end{thm}
If $\X$ is semistable, then $T(\X)$ is simply the set of singular points on the special
fibers of $\X$.
Note that in the proof of~\cite[Theorem~5.6]{ZhangAdmissible}, Zhang proves an analogue
(with $T(\X) = \emptyset$) of
Theorem~\ref{HtNonNeg} in the language of his admissible intersection theory.
Since we want to be able to compute lower bound on $\om^2$ for non-semistable $\X$, we
cannot use the admissible theory and have to work with hermitian line bundles throughout.
In order to use Theorem~\ref{HtNonNeg} to derive a lower bound on $\om^2$, we follow
Zhang's approach from~\cite{ZhangAdmissible}.
The idea is to show that under certain conditions, we have $\overline{\L_D}^2 \ge 0$.
If these conditions are satisfied and $(2g-2)D$ is a canonical $\Q$-divisor on $X$, then 
we can relate $\overline{\L_D}^2$ to
$\om^2$ and use this inequality to obtain lower bounds on $\om^2$.
As in Zhang's theory (cf.~\cite[Theorem~6.5]{ZhangPos}),  the crucial condition on 
$\overline{\L_D}$ is relative semipositivity, where we call a
hermitian line bundle {\em relatively semipositive} if its restriction to every
  irreducible vertical divisor has nonnegative arithmetic degree (see Definition~\ref{DefPos}).
The proof of the following result is similiar to the proof
of~\cite[Theorem~6.5]{ZhangPos}, but rather more complicated.
\begin{prop}\label{LDBound}
If $\overline{\L_D}$ is relatively semipositive and $D_\X\,\cap\, T(\X)=\emptyset$,
then we have $\overline{\L_D}^2\ge 0$.
\end{prop}
In Section~\ref{Vsec} we locally define certain vertical divisors $V_D$ and $U_D$ attached to $D$;
they are the main ingredients in the construction of $\overline{\L_D}$, see
Definition~\ref{LDDef}.
Moreover, we set 
\[
        \beta_D = \frac{1-g}{g}\O\left(2V_D+U_D\right)^2 + 2(\om\ldot \O(U_D)).
\]
\begin{thm}\label{thm-main}
Let $D \in \Div_\Q(X)$ be a $\Q$-divisor such that
$(2g-2) D$ is a canonical $\Q$-divisor on $X$ and such that
$D_\X$ satisfies $D_{\X} \,\cap \,T(\X) = \emptyset$.
If the hermitian line bundle $\overline{\L_D}$ is relatively semipositive,
then we have
\[
\om^2\ge \beta_D.
\]
\end{thm}
Note that a divisor $D$ as in Theorem~\ref{thm-main} always exists.
In order to derive a nontrivial lower bound on $\om^2$ from Theorem~\ref{thm-main} for a
given $\X$, we need to show that $\overline{\L_D}$ is relatively semipositive and that for
some choice of $D$ as in the statement of the theorem, we have $\beta_D \ge 0$.
\begin{thm}\label{thm-simple}
If $\X$ is minimal and all special fibers of $\X$ are reduced,
then the following are satisfied for every divisor $D \in \Div_\Q(X)$ of degree one:
\begin{enumerate}[\upshape (i)]
    \item $\overline{\L_D}$ is relatively semipositive;
    \item $\beta: = \beta_D$ does not depend on the choice of $D$;
    \item $\beta \ge 0$, with equality if and only if all special fibers of $\X$ are irreducible.
\end{enumerate}
\end{thm}
The proof of Theorem~\ref{thm-simple} essentially follows from a sequence of local
lemmas proved in Section~\ref{Vsec}.
We provide an explicit formula for $\beta$ in Lemma~\ref{betaformula}, making it very
simple to compute $\beta$ for a given minimal model $\X$ with reduced fibers.
As an immediate corollary we recover the following result from
\cite{ZhangAdmissible}, \cite{Moriwaki} and~\cite{Burnol}.
\begin{cor}\label{SemiPos}
    If $\X$ is semistable and minimal and has at least one reducible fiber, then there
    is an effectively computable positive lower bound on $\om^2$.
\end{cor}
In the semistable case lower bounds on $\om^2$ can be derived  by means of
the admissible intersection theory due to Zhang, cf.
\cite[Theorem~5.5]{ZhangAdmissible}.
This method requires the computation of admissible
Green's functions on the reduction graphs of the special fibers of $\X$
and has been employed by Abbes-Ullmo~\cite{AbbesUllmo} to find lower bounds for certain modular curves (see also Subsection~\ref{x1n}),
but such an approach does not work for non-semistable arithmetic
surfaces.
The positivity of $\om^2$ for non-semistable $\X$ with at least one reducible has been
proven by Sun in~\cite{sun}. 
However, his result is often not suitable for explicit computations of such bounds in practice,
since it requires computing a global semistable model over an extension of $K$.
We believe that for $D$ as in Theoerem~\ref{thm-main}, $\beta_D$ is a lower bound on $\om^2$ for all minimal arithmetic
surfaces, even those with non-reduced fibers.
Indeed, if we are given a minimal $\X$ having components of multiplicity $>1$,
we can still check whether Theorem~\ref{thm-main} is applicable.
As an example, we prove that the conditions of Theorem~\ref{thm-main} are satisfied for 
the minimal regular model $\sF^{\min}_p$ of the Fermat 
curve of prime exponent $p>3$ over the field of $p$-th cyclotomic numbers and that the
resulting lower bound is positive.
This does not follow from Theorem~\ref{thm-simple}, since the irreducible
components of $\sF^{\min}_p$ need not have multiplicity one.
\begin{thm}\label{thm-FpBound2}
The arithmetic self-intersection $\om^2$ of the relative dualizing sheaf on
$\sF^{\min}_p$ satisfies
\[
\om^2 \ge p \log p + \O(\log p).  
\]
\end{thm}
A more precise statement is provided in Theorem~\ref{FpBound}.
The paper is organized as follows: In Section~\ref{Vsec} we define the divisors $V_D$ and
$U_D$ locally and prove that they have certain properties with respect to the intersection
multitplicity.
We then switch to a global perspective in Section~\ref{shlb}, where we prove some
general results on hermitian line bundles. 
Section~\ref{hi} contains the definition of $\overline{\L_D}$ and the proof of
Theorem~\ref{HtNonNeg}.
The results of Sections~\ref{Vsec},~\ref{shlb} and~\ref{hi} are then used in
Section~\ref{Bds} to prove Proposition~\ref{LDBound} and Theorems~\ref{thm-main} and~\ref{thm-simple}.
At the end of that section, we also discuss a possible application of our results to the effective Bogomolov
conjecture.
In Section~\ref{sec:apps} we first use Theorem~\ref{thm-simple} to  prove an 
asymptotic lower bound for $\om^2$ on minimal
regular models of modular curves $X_1(N)$ for certain $N$.
Finally, we use Theorem~\ref{thm-main} to prove Theorem
\ref{thm-FpBound2}; here
we also compare the resulting lower bound to the upper bound computed by Curilla
and the first author in~\cite{CurillaKuehn}.
We would like to thank Ariyan Javanpeykar and David Holmes for a careful reading of the
manuscript and many helpful suggestions. 
We also thank Zubeyir Cinkir for helpful advice on the proof of
Lemma~\ref{HDGi} and Christian Curilla for drawing Figure~\ref{d1}. 
 
\section{Intersection properties of certain vertical divisors}\label{Vsec}
Let $\O$  be a strictly Henselian discrete valuation ring with field of fractions $K$. 
Let $\X_s$ be the special fiber of a
proper regular model $\X/\O$  of a smooth projective geometrically irreducible curve $X/K$ of genus $g>1$.
In this section we define  certain vertical divisors $V_D, U_D$ with support in the
special fiber $\X_s$ attached to $\Q$-divisors $D \in \Div_\Q(X)$  and study their properties.  
 
Suppose that, as a divisor on $\X$, the special fiber $\X_s$ is given by $\X_s = \sum^{r}_{i=1}b_i\Gamma_i$, where 
$\{\Gamma_1,\ldots,\Gamma_{r}\}$ is the set of irreducible components of $\X_s$ and the
$b_i$ are  positive integers. 
We fix a canonical divisor $\K$ on $\X$, and set
 \[
a_i = (\Gamma_i\ldot\K)
\]
for $i\in \{1,\ldots,r\}$, where $(\, \ldot\,)$ is the rational-valued intersection
multiplicity on $\X$.
Note that by the adjunction formula~\cite[Theorem~9.1.37]{LiuBook}, we have 
\[
  a_i=-\Gamma_i^2+2p_a(\Gamma_i)-2,
\]
where $p_a(\Gamma_i)$ is the arithmetic genus of $\Gamma_i$.
Given a nonzero $\Q$-divisor $D\in\Div_\Q(X)$, we denote the Zariski closure of $D$ in $\X$ by
$D_\X$.
\begin{prop}\label{Vexists}
	For every $D\in\Div_\Q(X)$ there exists a vertical $\Q$-divisor $V_D\in
        \Div_\Q(\X)$, which is 	unique up to addition of rational  multiples of $\X_s$, such that
\[
 		(D_\X+V_D\ldot\Gamma_i)=\frac{\deg(D)}{2g-2}a_i
\]
for all $i\in\{1,\ldots,r\}$. Moreover, the assignment 
\[
D \mapsto V_D \mod\, \X_s
\]
is linear in $D$.
\end{prop}
\begin{proof}  
The assignment 
\[
    E \mapsto \left( \left(\frac{\deg(D)}{2g-2}\,\K - D_\X\right)\ldot E\right) 
\]
defines a linear map on $Z^1(\X_s)_\Q=Z^1(\X_s)\, \otimes_\Z\,\Q$. By the non-degeneracy of the intersection pairing
on $Z^1(\X_s)_\Q$ modulo the entire fiber, 
this map is representable by a cycle $ V_D \in Z^1(\X_s)_\Q$.  As this  assignment is also a linear  map
in $D$, the two claims follow immediately. 
\end{proof} 
Proposition~\ref{Vexists} implies that we can extend any $\Q$-divisor
$D$ on $X$ to a $\Q$-divisor on $\X$ which satisfies the adjunction formula up to a factor
$\deg(D)$.
We can define a local pairing on
coprime divisors $E_1,E_2\in\Div_\Q(X)$ by
\begin{equation}\label{defpairing}
    [E_1,E_2]=(E_{1,\X}+V_{E_1}\ldot E_{2,\X}+V_{E_2}).
\end{equation}
\begin{cor}\label{pairing}
The pairing $[E_1,E_2]$ extends the local N\'eron pairing
(see~\cite[\S III.5]{LangArak}) to divisors of arbitrary degree. 
\end{cor}
To our knowledge, the pairing $[\cdot,\cdot]$ is the first extension of the local N\'eron
pairing to divisors of arbitrary degree that is not based on the
reduction graph as in~\cite{CRCapacity} or~\cite{ZhangAdmissible}.
\begin{cor}
Suppose that $D_l \in \Div_\Q(X)$ satisfies $(D_{l, \X}\ldot \Gamma_i) = \delta_{il}$ for all $i
\in\{1,\ldots,r\}$.
Then there exists a vertical divisor $V_l := V_{D_l}$ such that 
\begin{equation}\label{vldef}
(V_l\ldot\Gamma_i) = a'_i- \delta_{il},
\end{equation}
 where $a'_i=\frac{1}{2g-2}a_i$.
\end{cor}
Several applications will rely explicit formulas for the $V_l$.
We define a matrix $M = (m_{ij})_{i,j}$ as minus the intersection matrix of $\X_s$:
\[m_{i,j} = -(b_i\Gamma_i\ldot b_j\Gamma_j)\]
Let $M^+=(n_{ij})_{i,j}$ denote the Moore-Penrose pseudoinverse of $M$,
cf.~\cite[\S3]{CinkirPseudo}.
For fixed $l \in \{1,\ldots,r\}$ we define a vector 
\[
    c_l= (c_{l1},\ldots,c_{lr})^t = -M^+ w_l,
\]
where 
\[
    w_l = (w_{l1},\ldots,w_{lr})^t, \quad w_{lj}= b_ja'_j-\delta_{lj};
\]
here $\delta_{il}$ is the Kronecker delta.
\begin{prop}
If $l \in \{1,\ldots,r\}$, then a divisor $V_l$ satisfying~\eqref{vldef} exists.
Moreover, we have
\[
    V_l = \sum^r_{i=1}b_ic_{li}\Gamma_i.
\]
\end{prop}
\begin{proof}
It follows from~\cite[Corollary~9.1.10]{BLR} that a $\Q$-divisor $D_l$ satisfying 
$(D_{l, \X}\ldot \Gamma_i) = \delta_{il}$ exists for all $i \in\{1,\ldots,r\}$ .
The formula for $V_l$ is an immediate consequence of the relations 
\[
MM^+M=M\;\text{ and }\; M^+MM^+=M^+.
\]
\end{proof}
\begin{defn}
If $D \in \Div_\Q(X)$ has degree $d >0$, then we define a vertical $\Q$-divisor $U_D$ on $\X$ associated to $D$ as follows: 
For all $i \in \{1,\ldots,r\}$ we set
\[
\gamma_{D,i}=\frac{1}{d}\left(V_D^2 -(V_D-dV_i)\right)^2
\]
and define 
\[
        U_{D}=\sum^r_{i=1}\gamma_{D,i}\Gamma_i.
\]
\end{defn}
Our main motivation for this definition is the following formula for the intersection of
$U_D$ with horizontal divisors. It will play a crucial part in the proof of Theorem~\ref{HtNonNeg}.
If $D$ has degree~0, then we write
\begin{equation}\label{phiv}
    \Phi_\X(D) := V_D
\end{equation}
in accordance with the classical literature (see for
instance~\cite[Theorem~III.3.6]{LangArak}).
\begin{prop}\label{udhor}
Let $D \in \Div_\Q(X)$ have degree $d > 0$ and let 
 $E=\sum^e_{j=1}(P_j)$ be a nontrivial effective divisor on $X$, where
$P_j\in X(K)$. 
Then we have
\[
    d(E_\X\ldot U_D) = eV_D^2 -\sum^e_{j=1} \Phi_\X(dP_j-D)^2.
\]
Moreover, the association $D\mapsto U_D$ is linear in $D$.
\end{prop}
\begin{proof}
    For every $j \in \{1,\ldots,e\}$ there is an index $i_j \in \{1,\ldots,r\}$
    such that the section corresponding to $P_j$ intersects $\Gamma_{i_j}$ and
    does not intersect any other component.
Therefore 
\begin{equation}\label{ppb}
        \sum^e_{j=1}\Phi_{\X}(dP_j-D)^2=\sum^e_{j=1}
        \left(d^2V_{i_j}^2-2d(V_{i_j}\ldot V_D) +V_D^2\right)
        =-d\sum^e_{j=1}\gamma_{D,i_j}+eV_D^2.
\end{equation}
The first assertion follows from~\eqref{ppb} and
\[
        (E_\X\ldot  U_D)=\sum^e_{j=1} (P_{j,\X}\ldot  U_D)
                   =\sum^e_{j=1}\gamma_{D,i_j}.
               \]
The second assertion is trivial.
\end{proof}
We are now ready to define a local version of what will be our lower
bound on $\om^2$.
\begin{defn}
    If $D \in \Div_\Q(X)$ is a divisor of degree~1, 
    we define
    \[
        \beta_D = \frac{1-g}{g}\left(2V_D+U_D\right)^2 + 2(\K\ldot U_D).
    \]
\end{defn}
\begin{ex}\label{vuex}
Suppose that the special fiber of $\X$ consists of two irreducible components $\Gamma_1$ and $\Gamma_2$ of
multiplicity~1 and identical arithmetic genus
$p_a$ which intersect transversally in $s\ge 1$ points. 
Let $D=\frac{1}{2}D_1+\frac{1}{2}D_2$.
Then it is easy to see that we can take $V_D=0$ and
\[
    V_1 = \frac{1}{4s}\Gamma_1-\frac{1}{4s}\Gamma_2 = - V_2.
\]
This leads to
\begin{align*}
    U_{D_1} =& -\frac{1}{4s}\Gamma_1 +\frac{3}{4s}\Gamma_2 = - U_{D_2}\quad\text{and}\\
U_D = &\frac{1}{4s}\Gamma_1+\frac{1}{4s}\Gamma_2.
\end{align*}
A simple computation reveals
\[
\beta_D = \frac{1}{2s}(s+2p_a-2).
\]
\end{ex}
In order to show that (a global version of) $\beta_D$ indeed provides a non-trivial lower
bound for $\om^2$ in many situations, we first need to
prove some further intersection-theoretic properties of $U_D$.
To this end, we define a metrized graph $\mathcal{G}_{\X}$ as
follows: The vertex set of $\mathcal{G}_\X$ is given by
$\{\Gamma_1,\ldots,\Gamma_{r}\}$. There are no self-loops or multiple edges; two vertices $\Gamma_i$
and $\Gamma_j$ are connected by an edge if and only if $m_{ij}\ne 0$, in which case the
length of the edge is $-1/m_{ij}$.
We also need some facts about the matrix $M$ and its pseudoinverse $M^+$.
\begin{lemma}\label{mprops}
The following properties are satisfied:
\begin{enumerate}[(i)]
\item Both $M$ and $M^+$ are symmetric and positive semidefinite.
\item\label{Rows0} We have $\sum^r_{j=1}m_{ij}=\sum^r_{j=1}n_{ij}=0$ for all
$i\in\{1,\ldots,r\}$.
\item\label{Prod} We have 
  $\sum^r_{j=1}n_{ij}m_{jk}=-\frac{1}{r}+\delta_{kl}\textrm{  for all
  }i,k\in\{1,\ldots,r\}$.
\item\label{trace} We have 
$n_{ii}-\sum_{j,k}n_{ij}n_{kk}m_{jk}=\frac{\mathrm{Tr}(M^+)}{r}\textrm{  for all
  }i\in\{1,\ldots,r\}$.
\item $M$ is the discrete Laplacian matrix associated to
$\mathcal{G}_{\X}$.
\end{enumerate}
\end{lemma}
\begin{proof}
These properties are proved in~\cite{CinkirPseudo}.
\end{proof}
\begin{rk}
Note that when $\X$ is minimal and semistable, $\mathcal{G}_\X$ need not coincide with the reduction
graph $R(X)$ associated to
$X$ in~\cite{CRCapacity} and~\cite{ZhangAdmissible}.
For instance, suppose that $\X_s$ is given by two curves $\Gamma_1$ and $\Gamma_2$
intersecting transversally in $n$ points. 
In this case, $R(X)$ is the banana graph with $n$ edges of length~1, whereas 
$\mathcal{G}_\X$ is the complete graph with two vertices which are connected by a single edge of length $1/n$.
\end{rk}
From now on, we suppose that $D \in \Div_\Q(X)$ has degree one.
For $i \in \{1, \ldots, r\}$ we set
\begin{align*}
  v_i(D)&=(b_i\Gamma_i\ldot D_\X)\quad\text{and}\\
  w_i(D)&=b_i a'_i-v_i(D)=\frac{b_ia_i}{2g-2}-v_i(D).
\end{align*}
\begin{lemma}\label{HDGi}
Let $i \in \{1,\ldots,r\}$.
\begin{enumerate}[(a)]
\item  We have
  \[
    (U_D\ldot \Gamma_i)=
  -\sum_{j=1}^{r}n_{jj}m_{ij}+2v_i(D)-\frac{2}{r}.
  \]
\item If the special fiber of $\X$ is reduced, then
\[ 
(\K\ldot \Gamma_i)+2(D_\X\ldot \Gamma_i)-(U_{D}\ldot \Gamma_i)
\]
is independent of $D$ and nonnegative.
\end{enumerate}
\end{lemma}
\begin{proof}
Assertion (a) is an easy computation using Lemma~\ref{mprops}:
\begin{align*}
   (U_D\ldot \Gamma_i)&=-\sum^r_{l=1}\sum^r_{j=1}
    c_{lj}(w_{lj}-2w_j(D)) m_{li}\\
                  &=-\sum^r_{l=1}\left(\sum^r_{j=1}
  c_{lj}(2v_j(D)-a'_j-\delta_{lj})\right) m_{li}\\
  &=-\sum^r_{l=1}\sum^r_{j=1}\sum^r_{k=1} \left(2n_{kj}a'_kv_j(D)-n_{kj}a'_ka'_j-n_{kj}a'_k\delta_{lj}-2n_{kj}\delta_{kl}v_j(D)\right.\\
  &\qquad\qquad\qquad\left.+n_{kj}\delta_{kl}a'_j+n_{kj}\delta_{kl}\delta_{lj}\right) m_{li}\\
  &= - \sum_{j=1}^{r}n_{jj}m_{ij}+2\sum^r_{l=1}\sum^r_{j=1}n_{lj}v_j(D)m_{li}\\
  &= - \sum_{j=1}^{r}n_{jj}m_{ij}+2\sum^r_{j=1}v_j(D)\left(\delta_{ij}-\frac{1}{r}\right)\\
  &= - \sum_{j=1}^{r}n_{jj}m_{ij}+2v_i(D)-\frac{2}{r}.
\end{align*}
Now we turn to assertion (b) of the lemma and compute, using (a) and the adjunction
formula:
  \begin{equation}\label{indep}
      (\K\ldot \Gamma_i)+2(D_\X\ldot \Gamma_i)-(U_{D}\ldot \Gamma_i)
        =m_{ii}+2p_a(\Gamma_i)-2+\sum^r_{j=1}n_{jj}m_{ij}+\frac{2}{r}
\end{equation}
 
We deduce the first part of (b) and furthermore:
  \begin{equation}\label{resistform}
\sum_jn_{jj}m_{ij}=\sum_j(n_{jj}+n_{ii}-2n_{ij})m_{ij}-\frac{2}{r}+2
\end{equation}
Note that by~\cite[Lemma~4.1]{CinkirPseudo} we have
\[n_{jj}+n_{ii}-2n_{ij}=r(\Gamma_i,\Gamma_j),\]
where $r(\Gamma_i,\Gamma_j)$ is the effective resistance between the nodes
$\Gamma_i$ and $\Gamma_j$ if we consider the metrized graph $\mathcal{G}_{\X}$ as a
resistive electric circuit, where the resistance along an edge $e$
is given by the length $\ell(e)$.
Hence, using~\eqref{indep} and~\eqref{resistform}, it suffices to show 
\begin{equation}\label{Inequ}
  m_{ii}+\sum_jr(\Gamma_i,\Gamma_j)m_{ij}\ge0
\end{equation}
in order to prove assertion (b).
But we can rewrite the left hand side of~\eqref{Inequ} as
\[
 \sum_{j\ne i} m_{ij}(r(\Gamma_i,\Gamma_j)-1)
\]
because of Lemma~\ref{mprops}~\eqref{Rows0}.
A component $\Gamma_j$ can only contribute a negative summand to this sum if $m_{ij}\ne0$,
which means that the nodes on $\mathcal{G}_{\X}$ corresponding to $\Gamma_i$ and
$\Gamma_j$ are connected by an edge $e$ of length
$\ell(e)=-\frac{1}{m_{ij}}\le 1$.
But in this case the effective resistance $r(\Gamma_i,\Gamma_j)$ is bounded from
above by $\ell(e)$.
Hence all terms in the sum are nonnegative, proving~\eqref{Inequ} and thus the lemma.
\end{proof}
If $\X_s$ is reduced, then we can give a formula for the intersection of $U_D$ with a
canonical divisor.
This result will be important in order to show that for reduced $\X_s$, our lower bound $\beta_D$ does not depend
on $D$ and that $\beta_D$ is nonnegative if $\X$ is also minimal.
\begin{lemma}\label{udk}
Suppose that the special fiber of $\X$ is reduced.  
Then we have
\[
  (U_D\ldot \K) = -\sum^r_{i=1}V_i^2a_i.
\]
If, furthermore, $\X$ is minimal, then $(U_D\ldot \K)$ is nonnegative.
\end{lemma}
\begin{proof}
Since $D \mapsto U_D$ is linear in $D$ by Proposition~\ref{udhor},
we may assume that we have $D=D_l$ for some $l\in\{1,\ldots,r\}$, that is, $(D\ldot
\Gamma_j) = \delta_{lj}$ for all $j \in \{1,\ldots,r\}$.
By definition of $U_D$, we have
\[
 (U_{D}\ldot \K)=-\sum^r_{i=1}V_i^2a_i+2\sum^r_{i=1}(V_i\ldot V_l)a_i,
\]
so we have to show that
\begin{equation}\label{GiGt}
\sum^r_{i=1}(V_i\ldot V_l)a_i=0.
\end{equation}
Note that
\[
  (V_i\ldot V_l)= \sum^r_{j=1}c_{lj} \left(V_i\ldot \Gamma_j\right)=
\sum^r_{j=1}c_{lj}w_{ij},
\]                       
so that, using  $c^t_l=M^+ w^t_l$,
we find
\[
\sum^r_{i=1}(V_i\ldot V_l)a_i=-(2g-2)w_l^t M^+ \sum^r_{i=1}a'_iw_i.
\]
Therefore the proof of~\eqref{GiGt} follows from the fact that for each
$j\in\{1,\ldots,r\}$ we have
\[
\sum^r_{i=1} w_{ij}a'_i=a'_j\sum^r_{i=1}a'_i-a'_j=0,
\]
since $\sum^r_{i=1}a'_i=1$.  
If $\X$ is minimal, then the adjunction formula implies that 
\[
a_i=(\K\ldot \Gamma_i)=-\Gamma_i^2+2p_a(\Gamma_i)-2\ge0.
\]
for all $i$, which completes the proof of the lemma.
\end{proof}
The following result can be deduced immediately from Lemma~\ref{udk}.
\begin{cor}
Suppose that $\X$ is minimal with reduced special fiber.  
Then $\beta_D$ is nonnegative.
\end{cor}
It is natural to ask whether $\beta_D$ depends on $D$.
We will now show that this is not the
case when $\X$ has reduced special fiber; furthermore, we will provide a rather explicit formula for $\beta_D$.
\begin{lemma}\label{betaformula}
If $\X$ has reduced special fiber, then $\beta=\beta_D$ is independent of $D$.
More precisely, we have the following formula for $\beta$:
\begin{align*}
\beta&=\frac{4(g-1)}{gr}\mathrm{Tr}(M^+)+\frac{g-1}{g}\sum^r_{i=1}\sum^r_{j=1}n_{ii}n_{jj}m_{ij}+\frac{2(g-1)}{g}\sum^r_{i=1}a_in_{ii}-\frac{1}{g}\sum^r_{i=1}\sum^r_{j=1}a_ia_jn_{ij}.\\
\end{align*}
\end{lemma}
\begin{proof}
The proof is essentially a straightforward, but tedious computation using the properties of $M$ and
$M^+$ listed in Lemma~\ref{mprops}, so we do not present all details.
Suppose that $\X$ has reduced special fiber.
We first give an expression for $(K\ldot U_D)$. 
By Lemma~\ref{udk}, we have
\[
  (\K \ldot U_D) = -\sum^r_{i=1}V_i^2a_i.
\]
A simple computation shows that for a fixed $i \in \{1,\ldots,r\}$ we have
\[
V_i^2 = 2\sum_{j}a'_jn_{ij}-n_{ii}-\sum_{j,k}a'_ja'_kn_{jk};
\]
using Lemma~\ref{mprops} and $\sum_ia'_i=1$, this implies
\begin{equation}\label{kud}
(\K\ldot U_D)= (2g-2)\left(\sum_ia'_in_{ii} -\sum_{i,j}a'_ia'_jn_{ij}\right).
\end{equation}
Now we rewrite $(U_D+2V_D)^2$.
From the definition of $V_D$ we get
\begin{equation}\label{vdsq}
    V_D^2=2\sum_{i,j}a'_iv_j(D)n_{ij}-\sum_{i,j}a'_ia'_jn_{ij}-
    \sum_{i,j}v_i(D)v_j(D)n_{ij}.
\end{equation}
Next we compute, using Lemma~\ref{mprops} and omitting details:
\begin{align}
    (V_D\ldot U_D) =&\;\sum_i\left(2(V_D\ldot V_i)-V^2_i\right)w_i(D)\notag\\
                   =&\;\sum_{i,j,k}a'_ja'_kv_i(D)n_{jk}-2\sum_{i,j,k}a'_kv_i(D)v_j(D)n_{jk}+2\sum_{i,j}v_i(D)v_j(D)n_{ij}\notag\\
                   &\;+\sum_{i,j,k}a'_ia'_ja'_kn_{jk}
    -2\sum_{ij}a'_ia'_jn_{ij}+\sum_ia'_in_{ii}-\sum_iv_i(D)n_{ii}\notag\\     
                   =&\; 2\sum_{i,j}v_i(D)v_j(D)+\sum_ia'_in_{ii}-2\sum_{i,j}a'_iv_j(D)n_{ij}-\sum_iv_i(D)n_{ii}
\label{vdud}.
\end{align}
The computation of $U_D^2$ more complicated than the previous one, so we
only provide a rough sketch:
\begin{align}
    U_D^2 =&-4\sum_{i,j}(V_D\ldot V_i)(V_D\ldot V_j)m_{ij}+4\sum_{i,j}(V_D\ldot
    V_i)V_j^2m_{ij}-\sum_{i,j}V_i^2V_j^2m_{ij}\notag\\
          =&\;
    4\sum_{i,j,k}v_i(D)n_{ij}n_{kk}m_{jk}-\sum_{i,j}n_{ii}n_{jj}m_{ij}-4\sum_{i,j,k,l}v_i(D)v_l(D)n_{ij}n_{kl}m_{jk}\notag\\
          =&\;    4\sum_{i,j,k}v_i(D)n_{ij}n_{kk}m_{jk}-\sum_{i,j}n_{ii}n_{jj}m_{ij}-4\sum_{i,j}v_i(D)v_j(D)n_{ij}
\label{udsq}.
\end{align}
Combining~\eqref{vdsq},~\eqref{vdud} and~\eqref{udsq}, we find
\[
(U_D+2V_D)^2=4\sum_ia'_in_{ii}-\sum_{i,j}n_{ii}n_{jj}m_{ij}-4\sum_{i,j}a'_ia'_jn_{ij}+4\sum_{i,j,k}v_i(D)n_{ij}n_{kk}m_{jk}-4\sum_iv_i(D)n_{ii}.
\]
Part~(\ref{trace}) of Lemma~\ref{mprops} implies
\begin{equation}\label{ud2vdsq}
    (U_D+2V_D)^2=4\sum_ia'_in_{ii}-\sum_{i,j}n_{ii}n_{jj}m_{ij}-4\sum_{i,j}a'_ia'_jn_{ij}-\frac{4\mathrm{Tr}(M^+)}{r},
\end{equation}
which does not depend on $D$.
The result now follows from~\eqref{kud} and~\eqref{ud2vdsq}.
\end{proof}
\begin{ex}\label{betaex}
Keeping the notation of Example~\ref{vuex}, Lemma~\ref{betaformula} 
immediately implies that 
\[
    \beta = \frac{1}{2s}(s+2p_a-2).
\]
\end{ex}
\begin{rk}\label{constants}
    Note that the first two terms in the formula for $\beta$ given in Lemma~\ref{betaformula} only depend
on $M$, so they only depend on the combinatorial configuration of $\X_s$.
The last two terms, however, do depend on the arithmetic genera of the
irreducible components; more precisely, we have
\[
    \frac{2(g-1)}{g}\sum^r_{i=1}a_in_{ii}-\frac{1}{g}\sum^r_{i=1}\sum^r_{j=1}a_ia_jn_{ij}=\frac{2g-2}{g}\sum_ia_i\left(n_{ii}-\sum_ja'_jn_{ij}\right).
\]
Therefore, if $\X$ is semistable and minimal, then $\beta$ can be viewed as an
invariant of the polarized metrized graph $(R(X), {\bf q})$ associated to $\X_s$, where the
polarization ${\bf q}$ assigns to each component its arithmetic genus,
see~\cite[\S4]{CinkirBogomolov}.
\end{rk}
\begin{rk}\label{eps-beta-g2}
It seems worthwile to relate $\beta$ to other invariants of $(R(X), {\bf q})$, such
as Zhang's invariants $\eps$ (called $r$ in~\cite{ZhangAdmissible}), $\varphi$ and
$\lambda$. 
See~\cite{CinkirBogomolov} for definitions of and some relations between these invariants.
If $X$ is hyperelliptic, then it would also be interesting to compare $\beta$ to the
invariant $\chi$ studied, for instance, in~\cite{deJongSymmetric}.
Because of its potential relevance for an effective version of the Bogomolov conjecture
for curves over number fields (see Remark~\ref{bogomolov}), it is especially interesting
to compare $\beta$ to $\eps$.
We have computed $\beta$ for all semistable reduction types of genus~2 curves.
Table~\ref{compare} contains the values of $\beta$ and the values of $\eps$, computed by
de Jong, cf.~\cite[\S2]{deJongGenus2}.
We find that in genus~2, we always have $\beta \le \eps$.
\begin{center}
\begin{table}\begin{tabular}{|l|l|l|}
\hline
Type & $\eps$ & $\beta$ \\
\hline
I & 0 & 0 \\
II$(a)$ & $a$ & $a-1$ \\
III$(a)$ & $\frac{1}{6}a$ & $\frac{1}{6}a-\frac{1}{6a}$ \\
IV$(a,b)$ & $a+\frac{1}{6}b$ &$a+\frac{1}{6}b-\frac{1}{6b}$  \\
V$(a,b)$ & $\frac{1}{6}(a+b)$ & $\frac{1}{6}(a+b)-\frac{1}{6a}  -\frac{1}{6b}  $ \\
VI$(a,b,c)$ & $a+\frac{1}{6}(b+c)$ &  $a+\frac{1}{6}(b+c)-\frac{1}{6b}  -\frac{1}{6c}$ \\
VII$(a,b,c)$ & $\frac{1}{6}(a+b+c)+\frac{1}{6}\frac{abc}{ab+ac+bc}$ &
$\frac{1}{6}(a+b+c)+\frac{1}{6}\frac{abc}{ab+ac+bc} - \frac{a^2b + a^2c +
ab^2 + 6abc + ac^2 + b^2c + bc^2}{6(ab+ac+bc)^2}$   \\
\hline
\end{tabular}
\caption{$\;\;$The invariants $\eps$ and $\beta$ in genus~2}\label{compare}\end{table}
\end{center}
\end{rk}
\section{Semipositive hermitian line bundles}\label{shlb} 
Let $K$ be a number field with ring of integers $\O_K$ and let $\X$ be a regular
arithmetic surface over $\O_K$ whose generic fiber $X=\X_K$ is a smooth projective
geometrically irreducible curve of genus $g>1$.
In this section we prove several general lemmas about certain hermitian
line bundles on $\X$.
All of these will be used in the proof of Proposition~\ref{LDBound}.
Several results of this section are quite similar to results from~\cite{ZhangPos}.
We start with a number of definitions.
\begin{defn}
    Let $\overline{\L}$ be a hermitian line bundle on $\X$.
If $E$ is an irreducible effective divisor on $X$ with Zariski closure
$\mathcal{E}$, then the {\em height} of
$E$ with respect to $\overline{\L}$ is defined by 
\[
h_{\overline{\L}}(\mathcal{E})=\frac{(\overline{\L}\ldot
\Oa(\mathcal{E}))}{[K:\Q]\deg(\mathcal{E}_K)},
\]
where the metric on $\O(\mathcal{E})$ is admissible in the Arakelov-theoretic
sense and $(\cdot\ldot\cdot)$ is the arithmetic intersection pairing on $\X$, see for
instance~\cite{Soule}.
We extend this to arbitrary effective divisors on $\X$ by linearity.
\end{defn}
\begin{defn}\label{DefPos}
We say that a hermitian line bundle $\overline{\L}$ is {\em relatively
semipositive} if it has nonnegative intersection with all irreducible vertical
components of $\X$. If $\overline{\L}$ has nonnegative (resp. positive)
intersection with all irreducible horizontal divisors on $\X$, then we call
$\overline{\L}$ {\em horizontally semipositive} (resp. {\em horizontally positive}).
\end{defn}
\begin{defn}\label{DefEff}
Let $\overline{\L}$ be a hermitian line bundle on $\X$. 
We call a nonzero section $s$ of $\overline{\L}$ {\em effective} (resp. {\em strictly
effective}) if $\|s\|_\mathrm{sup}\le1$ (resp. $\|s\|_\mathrm{sup}<1$). 
We say that $\overline{\L}$ is {\em ample} if $\L$ is ample and $H^0\left(\X,\L^{\otimes
n}\right)$ has a basis consisting of strictly effective sections for $n\gg0$.
If this holds for $n=1$, then we call $\overline{\L}$ {\em very ample}.
\end{defn}
\begin{lemma}\label{ImitateZhang}
Let $\overline{\L}$ be a hermitian line bundle on $\X$. For any hermitian line bundle
$\overline{\M}$ on $\X$ and $a,b\in\N$ we set
\begin{align*}
\overline{\M}_{a,b} = \overline{\M}^{\otimes a}\otimes \overline{\L}^{\otimes b}. 
\end{align*}
If $\overline{\L}^2<0$, then there exists no ample hermitian line bundle
$\overline{\M}$ on $\X$ with the following property:
For all $a,b \in \mathbb{N}$ such that $\overline{\M}_{a,b}^2
>0$, we also have 
\begin{align}\label{fast-positiv}
\left(\overline{\L}\ldot \overline{\M}_{a,b}\right)\ge 0. 
 \end{align}
 \end{lemma}
 \begin{proof}
This proof is somewhat similar to the first part of the proof of~\cite[Theorem~6.3]{ZhangPos}. 
Suppose that $\overline{\L}^2<0$ and that $\overline{\M}$ is a hermitian line
bundle on $\X$
satisfying~\eqref{fast-positiv}.
Since $\overline{\M}$ is ample, it has  positive  arithmetic self-intersection by~\cite[Theorem~1.3]{ZhangPos}. 
Therefore, if $p(t)$ denotes the polynomial
\[
        p(t)= 
        \overline{\L}^2+2(\overline{\L}\ldot \overline{\M})t+\overline{\M}^2t^2,
\]
then there is a positive real number $t_0$ satisfying $p(t_0)=0$ and $p(t)>0$ for every $t>t_0$.
Let $a,b\in\N$ such that $a/b>t_0$. Then we find 
\[
\overline{\M}_{a,b}^2=b^2p(a/b)>0.
\]
By~\eqref{fast-positiv}, we know that
\[
\frac{1}{b} \left(\overline{\L}\ldot \overline{\M}_{a,b}\right) =
\overline{\L}^2+\frac{a}{b}(\overline{\L}\ldot \overline{\M})\ge0.
\]
In particular, our assumption that $\overline{\L}^2<0$ implies 
\begin{equation}\label{v2-lm}
        (\overline{\L}\ldot \overline{\M})>0,
\end{equation}
and also, since  $a/b$ can be arbitrary  close to $t_0$,
\begin{equation}\label{v2-lllmt0}
\overline{\L}^2+(\overline{\L}\ldot \overline{\M})t_0\ge0.
\end{equation}
Now we can derive a contradiction as in the proof of~\cite[Theorem~6.3]{ZhangPos}. Namely, combining $p(t_0)=0$ and~\eqref{v2-lllmt0} implies
\[
        (\overline{\L}\ldot \overline{\M})t_0+\overline{\M}^2t_0^2\le0.
\]
But using $\overline{\M}^2>0$  and~\eqref{v2-lm}, we see that this is impossible.
\end{proof}
The following result provides us with a method to show that two hermitian line bundles on $\X$
        have nonnegative intersection.
\begin{lemma} \label{Mlist}
Let $\overline{\L}$ and $\overline{\M}$ be hermitian line bundles on $\X$.
 If $\overline{\M}$ has an effective global section $s$ such that
 $h_{\overline{\L}}({\rm{div}(s)^{hor}})\ge 0$ and $\overline{\L}$ is relatively semipositive, then
$(\overline{\L}\ldot \overline{\M})\ge 0$.
\end{lemma}
\begin{proof} 
  According to~\cite[\S3.2.2]{BGS}, we have 
        \[\left(\overline{\L}\ldot \overline{\M}\right) =
\frac{\left(\O({\rm{div}(s)^{ver}})\;.\;\overline{\L}\right)}{[K:\Q]}+
          h_{\overline{\L}}({\rm{div}(s)^{hor}}) - \int \log\|s\|_{\sup}\;
      c_1(\overline{\L}).\]
Since  $\rm{div}(s)^{ver}$ is effective, the claim follows.
\end{proof}
Suppose we want to show that a hermitian line bundle $\overline{\L}$ on $\X$ satisfies
$\overline{\L}^2\ge0$. 
By Lemma~\ref{ImitateZhang}, it suffices to find some ample hermitian line bundle
$\overline{\M}$ on $\X$ such that under the assumption $\overline{\L}^2<0$ we have
$\left(\overline{\L}\ldot \overline{\M}_{a,b}\right)\ge 0$ whenever $\overline{\M}_{a,b}^2>0$.
If $\overline{\L}$ is horizontally semipositive, then we can take any ample
$\overline{\M}$
and any effective section $s$ of $\overline{\M}_{a,b}^{\otimes n}$ (which exists for $n\gg0$ by 
\cite[Theorem~2.1]{ZhangPos}, see~\cite[\S8]{ZhangPos}) and apply Lemma~\ref{Mlist}.
However, in the proof of Proposition~\ref{LDBound} we will apply Lemma~\ref{Mlist} to a hermitian line bundle
$\overline{\L}=\overline{\L_D}$
which is not in general horizontally semipositive, but only satisfies
$h_{\overline{\L_D}}(E)\ge 0$ if the Zariski closure $E_\X$ avoids a certain
finite set of points, 
 so we have to be more careful with our choice of $s$.
Lemma~\ref{AlphaExists} below tells us that, under the hypothesis $\overline{\L}^2<0$, we can find 
$\overline{\M}$ such that for $n\gg0$ there are many effective sections of
$\overline{\M}_{a,b}$ whenever $\overline{\M}_{a,b}^2>0$. 
Intuitively, it should be possible to find an effective section $s$ avoiding a finite set
of points if $\overline{\M}_{a,b}$ has enough effective sections.
Lemma~\ref{Mample} makes this intuition precise.
\begin{lemma}\label{AlphaExists}
Suppose $\overline{\L}$ is a relatively semipositive hermitian line bundle on
$\X$ such that $\deg(\L)>0$ and $\overline{\L}^2<0$. Then there exists an ample
hermitian line bundle $\overline{\M}$ on $\X$ such that for all $a,b \in
\mathbb{N}$ satisfying 
\begin{align*}
\overline{\M}_{a,b}^2 = \left(\overline{\M}^{\otimes a}\otimes \overline{\L}^{\otimes b}\right)^2 >0,  
\end{align*}
the lattice
$H^0\left(\X,{\M}_{a,b}^{\otimes n}\right)$ has a basis consisting of strictly effective global sections for some $n=n(a,b)\gg0$.
\end{lemma}
\begin{proof}
Let $\overline{\M}$ be an ample hermitian line bundle on $\X$. We will scale
$\overline{\M}$ by $\alpha\in\Q_{\ge0}$ such that the lemma holds for
$\overline{\M}(\alpha)$. By~\cite[Theorem~1.5]{ZhangPos}, it suffices to show
that $\overline{\M}_{a,b}(\alpha)^2>0$ implies that  $\overline{\M}_{a,b}(\alpha)$ is horizontally positive, since relative semipositivity is automatic.
Let $p_\alpha(t)$ denote the polynomial
\[
        p_\alpha(t)= 
        \overline{\L}^2+2\left(\overline{\L}\ldot
    \overline{\M}(\alpha)\right)t+\overline{\M}(\alpha)^2t^2.
\]
As in the proof of Lemma~\ref{ImitateZhang}, there is some positive real number $t_0=t_0(\alpha)$ such that $p_\alpha(t_0)=0$ and $p_\alpha(t)>0$ for all $t>t_0$.
Now let $m_\L=-\inf_{D}h_{\overline{\L}}(D)$ and $m_\M=\inf_{D}h_{\overline{\M}}(D)$,
where we take the infima over all irreducible divisors $D$ on $X$. If $m_\L\ge0$, then we
can take $\alpha=0$, so we may assume that $m_\L<0$.
Let $D$ be some irreducible divisor on $X$. We will construct $\alpha$ such that
\begin{equation}\label{Ht0}
h_{\overline{\M}_{a,b}(\alpha)}(D)\ge0
\end{equation}
whenever $a/b>t_0(\alpha)$, which will prove the lemma.
Because of
\begin{align*}
h_{\overline{\M}_{a,b}(\alpha)}(D)&=ah_{\overline{\M}(\alpha)}(D)+bh_{\overline{\L}}(D)\\
                                                                                                                                 &\ge
am_\M+a\alpha-bm_\L,
\end{align*}
we need a nonnegative $\alpha$ such that $\frac{a}b\ge t_0(\alpha)\Rightarrow\frac{a}b\ge
\frac{m_\L}{\alpha+m_\M}$, so $\alpha$ must satisfy
\[
        t_0(\alpha)\ge\frac{m_\L}{\alpha+m_\M}. 
\]
Hence~\eqref{Ht0} is easily seen to follow from
\begin{equation}\label{AlphaIneq}
(\alpha+m_\M)r(\alpha)\ge
m_\L\left(\overline{\M}^2+2\alpha\deg(\M)\right)+(\alpha+m_\M)\left((\overline{\M}\ldot
\overline{\L})+\alpha\deg(\L)\right),
\end{equation}
where 
\[
        r(\alpha)=\sqrt{\alpha^2\deg(\L)^2+2\alpha\deg(\M)\left((\overline{\M}\ldot
        \overline{\L})-\overline{\L}^2\right)+(\overline{\M}\ldot
\overline{\L})^2-\overline{\L}^2\overline{\M}^2}.
\]
Note that $r(\alpha)$ is real since $p_\alpha$ always has real roots. Hence the left hand side of~\eqref{AlphaIneq} is always nonnegative. If the right hand side is negative for some $\alpha$, then~\eqref{AlphaIneq} holds and we are done, so we may assume that the right hand side is also nonnegative.
We find that~\eqref{AlphaIneq} holds if and only if $w(\alpha)\ge0$, where $w$ is a cubic polynomial in $\alpha$, obtained by subtracting the square of the right hand side of~\eqref{AlphaIneq} from the square of the left hand side. The leading coefficient of $w$ is 
\[
-2\deg(\M)(\overline{\L}^2+2m_\L\deg(\L))
\]
which is positive by our assumptions on $\overline{\L}$. Hence~\eqref{AlphaIneq} holds for $\alpha\gg0$.
\end{proof}
\begin{lemma} \label{Mample}
        Let $\overline{\M}$ be an ample hermitian line bundle on $\X$ and let
$P_1,\ldots,P_r$ be closed points on $\X$. 
Then for some $n\gg 0$ there exists a strictly effective global section $s$ of
$\M^{\otimes n}$ such that $s(P_i)\ne 0$ for $i=1,\ldots,r$. 
\end{lemma} 
\begin{proof} Without loss of generality we assume that $\overline{\M}$ is very ample.
Since $\overline{\M}$ is very ample as a hermitian line bundle, the lattice $H^0(X,\M)$ is spanned
by strictly effective sections $s_1,\ldots,s_m$.
Moreover, since $\M$ is also very ample in the geometric setting, it is globally
generated, i.e., for each point $P_i$ there exists at least one section $s_j$
such that $s_j(P_i) \ne 0$. Let
$\{s_1,\ldots,s_k\}$ be a minimal set of such sections. 
Then for every even $n$ the global section 
$s'=\sum^k_{j=1} s_j^n$  of $\M^{\otimes n}$ 
satisfies  $s'(P_i)\ne 0$ for all $i=1,\ldots,r$. 
But now for some even $n\gg0$  
 the global section $s'=\sum^k_{j=1}  s_j^n$ 
also satisfies the condition
\[\|s'\|_{\rm sup} \le  k \max_j \|s_j\|^n  <1.\]
\end{proof}
\section{Heights and intersections}\label{hi}
We keep the notation of the previous section.
If $F$ is a finite extension of $K$, then we let
$\varphi_F=\mathrm{pr}_1\circ\pi_F:\X^F\to\X$, where
$\pi_F:\X^F\to\X\times\O_F$ denotes the minimal desingularization. 
For the definition of semistable arithmetic surfaces we refer to~\cite{LiuStable}; in particular, we do
not require a semistable arithmetic surface to be minimal.
\begin{lemma}\label{Liul}(Liu,~\cite{LiuStable})
  There exists a finite extension $F_0/K$ such that

  $\X^{F}$ is semistable for every finite extension $F/F_0$.  
\end{lemma}
\begin{defn}\label{DefT}
Let $F_0/K$ be as in Lemma~\ref{Liul}.
We denote the smooth locus of $\X^{F_0}$ by $\X^{F_0}_{\mathrm{sm}}$ and we denote the
exceptional locus of $\varphi_{F_0}$ by $\mathrm{Exc}(\varphi_{F_0})$.
With this notation we define 
\[
  T(\X)=\varphi_{F_0}\left(\X^{F_0}\setminus\X^{F_0}_{\mathrm{sm}}\;\cup\;
  \mathrm{Exc}(\varphi_{F_0})\right).
  \]
\end{defn}
\begin{rk}
  If $\X$ is semistable, then we have
  $T(\X)=\X\setminus\X_{\mathrm{sm}}$.
\end{rk}
\begin{lemma}\label{FiniteT}
Every irreducible divisor $E$ on $X$ such that $\supp(E_\X)\cap T(\X)=\emptyset$ satisfies 
$\varphi^*_F E_\X=E_{\X^F}$ for every finite extension
$F/F_0$.
\end{lemma}
\begin{proof}
Let $E$ be an irreducible  divisor on $X$ whose closure $E_\X$ does not
contain an element of $T(\X)$ in its support and let $F/K$ be a finite
extension containing $F_0$.
Note that if $\varphi^*E_\X\ne E_{\X^F}$, then there is an
irreducible component $\Gamma\subset\mathrm{Exc}(\varphi_F)$ such that
$\varphi_F(\Gamma)\in\supp(E_\X)$.
But this means that either $\Gamma\subset\mathrm{Exc}(\varphi_{F_0})$, implying
that $\varphi_{F}(\Gamma)\in T(\X)$, or $\Gamma$ is contracted
to a point by the desingularization morphism
$\X_{F}\to\X_{F_0}\times\O_F$.
In this case $\Gamma$ maps to a singular point of
$\X_{F_0}\times\O_F$, whence $\varphi_{F}(\Gamma)\in T(\X)$.
\end{proof}
 
Let $h_{\NT}$ denote the N\'eron-Tate height on the Jacobian $J$ of $X$ with respect to the symmetrized theta divisor $\Theta+[-1]^*\Theta$.
       For each divisor $D\in\Div_\Q(X)$ of degree one, let $j_D:X\hookrightarrow J$ be
       the embedding which maps a point $Q\in X$ to the class of $Q-D$. 
\begin{defn}\label{LDDef}
   Let $D \in \Div_\Q(X)$ have degree one. 
    For every non-archimedean place $v$ of $K$ we define $D_v = D \times K^{\mathrm{nr}}_v$,
    where $K^{\mathrm{nr}}_v$ is the maximal unramified extension of the completion of $K$ at 
$v$ and we fix the proper regular
model $\X\times\O^{\mathrm{nr}}_v$ of $X$ over  the ring $\O^{\mathrm{nr}}_v$ of
    integers of $K^{\mathrm{nr}}_v$. 
    With these choices, we define vertical divisors
    \[
        V_D=\sum_v V_{D_v}
    \]
    and 
    \[
        U_D=\sum_v U_{D_v}
    \]
    on $\X$,  where both sums are over all non-archimedean places of $K$.
    Moreover, we define a hermitian line bundle $\overline{\L_D}$ on $\X$ by
\[
\overline{\L_D}=\om\otimes\Oa(2D_\X)\otimes\O(U_D)^{-1}\otimes\Oa(-a\X_\infty),
\]
where $a=\Oa(D_\X)^2-\O(V_D)^2\in\R$.
Here $\Oa(D_\X)^2$ is the self-intersection of
$\O(D_\X)$, equipped with the Arakelov metric, and
$\Oa(-a\X_\infty)=(\O_\X,|\cdot|e^a)$.
Finally, we set
\[
    \beta_D = \sum_v \beta_{D_v}.
\]
\end{defn}
Unless otherwise stated, all metrics will be Arakelov metrics (so that the Arakelov
adjunction formula holds, cf.~\cite[\S IV.5]{LangArak}), except for vertical line bundles,
which are equipped with the trivial metric. 
Now we prove Theorem~\ref{HtNonNeg}, stating that the height with respect to $\overline{\L_D}$ is closely
related to the N\'eron-Tate height on the Jacobian of $X$.
\begin{proof}[Proof of Theorem~\ref{HtNonNeg}]
Let $E$ be an irreducible divisor on $X$ such that
$\supp(E_\X)\cap T(\X)=\emptyset$.
Let $F_0$ be as in Lemma~\ref{Liul} and let $F/K$ be a finite extension
containing $F_0$ such that $E$ has pointwise $F$-rational support.
Let $e=\deg(E)$ and $E=\sum^e_{j=1}(P_j)$, where $P_j\in X(F)$. 
By Lemma~\ref{FiniteT} we have 
\begin{equation}\label{FPullBack}
  \varphi_F^*E_\X=E_{\X^F}\;\textrm{ and }\;
  \varphi_F^*D_\X=D_{\X^F}.  
\end{equation}
Hence we get, using~\cite[Theorem~III.4.5]{LangArak},
\begin{equation}\label{phipb}
  \Phi_{\X^F}(E-eD)=\varphi_{F}^*(\Phi_\X(E-eD)),
\end{equation}
where, if $Z \in \Div_\Q(X)$ has degree~0, $\Phi_\X(Z) \in \Div_\Q(\X)$ is a vertical 
divisor such that $Z_\X+\Phi_\X(Z)$ has
trivial intersection multiplicity with all vertical divisors on $\X$.
In our notation, we have $\Phi_\X(Z) = \Phi_{\X\times\O^{\mathrm{nr}}_v}(Z_v)= \sum_v V_{Z_v}$, see~\eqref{phiv}.
Expanding the left hand side of~\eqref{phipb}, we find
\[
    \Phi_{\X^F}(E-eD)=\sum^e_{j=1} \Phi_{\X^F}(P_j-D)
\]
and Proposition~\ref{udhor} implies that
\begin{equation}\label{EPhi}
        (\Oa(E_\X)\ldot  \O(U_D))
        =e\O(V_D)^2-\sum^e_{j=1}\O\left(\Phi_{\X^F}(P_j-D)\right)^2.
\end{equation}
This allows us to compare $h_{\overline{\L_D}}$ to $h_\mathrm{NT}$. We will use
the Hodge Index Theorem on arithmetic surfaces due to Faltings and Hriljac (see for instance~\cite[\S III.5]{LangArak})
which implies that if $P\in J(K)$, then we have 
\begin{equation}\label{FaltHril}
h_{\NT}(P) [K:\Q]=-\left(\Oa(Z_\X)\otimes\O(\Phi_\X(Z))\ldot
\Oa(Z_\X)\right)=-\Oa(Z_\X)^2+\O(\Phi_\X(Z))^2
\end{equation}
for every  divisor $Z$ of degree zero on $X$ such that $Z$ represents $P$. 
Setting $a'=-\Oa(D_\X)^2$, we get
\begin{align*}
        &\sum^e_{j=1}h_{\NT}(j_D(P_j))[K:\Q]
        =\sum^e_{j=1}
        \left(-\Oa\left(P_{j,\X^F}-D_{\X^F}\right)^2+\O\left(\Phi_{\X^F}(P_{j}-D)\right)^2\right)\\
        &=\sum^e_{j=1}\left(
        -\Oa\left(P_{j,\X^F}\right)^2+2\left(\Oa(P_{j,\X^F})\ldot
        \Oa(D_{\X^F})\right)+\O\left(\Phi_{\X^F}(P_{j}-D)\right)^2\right)-e
        a'\\
        &=\sum^e_{j=1}\left(
        \left(\Oa(P_{j,\X^F})\ldot 
        \om_{\X^F}\otimes
        \Oa(2D_{\X^F})\right)+\O\left(\Phi_{\X^F}(P_{j}-D)\right)^2\right)-e
        a'\\
        &=(\Oa(E_{\X^F})\ldot 
        \om_{\X^F}\otimes
        \Oa(2D_{\X^F}))+\sum^e_{j=1}
        \O\left(\Phi_{\X^F}(P_{j}-D)\right)^2-e a'\\
        &=(\Oa(E_{\X})\ldot  \om\otimes \Oa(2D_{\X}))-(\Oa(E_\X)\ldot 
        \O(U_D))-e a'+e\O(V_D)^2\\
        &=(\Oa(E_\X)\ldot  \overline{\L_D})\\
        &=e[K:\Q] h_{\overline{\L_D}}(E).
\end{align*}
Here  the first equality holds by~\eqref{FaltHril}, the third equality holds because of
the Arakelov adjunction formula (see~\cite[\S IV.5]{LangArak}) and the fifth equality
holds because of~\eqref{FPullBack},~\eqref{EPhi} and because, by assumption,
$E_{\X^l}$ does not intersect any vertical divisors contracted by $\varphi_F$. 
The first assertion of the proposition is now immediate since the N\'eron-Tate height only takes
nonnegative values. 
The second assertion follows if we put $E=(P)$, where $P \in X(K)$.
\end{proof}
\begin{rk}
If $X$ is a smooth projective geometrically irreducible  curve  defined over an archimedean 
local field
and $E_1,E_2 \in \Div(X)$ have disjoint support, then we set  $[E_1,E_2] = (E_1, E_2)_a$, 
where the latter denotes the
admissible pairing on $X$, see \cite[\S4.5]{ZhangAdmissible}. 
Now suppose that $X$ is defined over a number field $K$.
We can use~\eqref{defpairing} and Corollary~\ref{pairing} to define a pairing on 
divisors $E_1, E_2$ on $X$ with disjoint support as
\[
  [E_1,E_2]=\sum_v[E_{1,v},E_{2,v}]_v,
\]
where $E_{i,v} = E_{i}\times_K K_v$ for archimedean $v$ and the sum is over all places of $K$.
This global pairing has the following  properties, which may be of independent
interest:
\begin{itemize}
  \item[(i)] $[\cdot,\cdot]$ is bilinear and symmetric.
  \item[(ii)] $[E_1,\mathrm{div}(f)]=0$ for any $f\in K(X)^*$. Hence $[\cdot,\cdot]$
    induces a well-defined pairing on divisor classes.
  \item[(iii)] If $\deg(E_1)=\deg(E_2)=0$, then we have
    $[E_1,E_2]=-(E_1,E_2)_{\NT}$, where the latter is the N\'eron-Tate height
    pairing.
  \item[(iv)] If $E_1$ and $E_2$ are canonical divisors on $X$, then we have
    $[E_1,E_2]=\om^2$. 
\end{itemize}
\end{rk}
\section{Proofs of Proposition~\ref{LDBound}, Theorem~\ref{thm-main} and Theorem~\ref{thm-simple}}\label{Bds}
We finally get to our original problem, namely the derivation of lower bounds on $\om^2$.
We first prove Proposition~\ref{LDBound}.
If $\overline{\M}$ is a hermitian line bundle on $\X$ and $a,b$ are positive integers, then we set
\[
\overline{\M}_{a,b}=\overline{\M}^{\otimes a}\otimes\overline{\L_D}^{\otimes b}.
\]
We want to use Lemma~\ref{ImitateZhang} to prove Proposition~\ref{LDBound}, so we
need to show that under the hypothesis $\overline{\L_D}^2<0$ there is some
hermitian line bundle $\overline{\M}$ on $\X$ with positive
self-intersection such that $\overline{\M}_{a,b}^2>0$ implies
$\left(\overline{\L_D}\ldot \overline{\M}_{a,b}\right)\ge0$.
\begin{prop}\label{Mexists}
Suppose that $\overline{\L_D}^2<0$ and that $D_\X\;\cap\; T(\X)=\emptyset$.
Then there exists an ample hermitian line bundle $\overline{\M}$ on $\X$ such that for any positive
integers $a,b$ the following condition is satisfied: If $\overline{\M}_{a,b}^2> 0$, then
there exists a positive integer $n(a,b)$ such that $\overline{\M}_{a,b}^{\otimes n(a,b)}$
has an effective section $s$ satisfying
$h_{\overline{\L_D}}({\rm{div}(s)^{hor}})\ge0$.
\end{prop}
\begin{proof} 
It follows from Lemma~\ref{AlphaExists} that there is an ample hermitian line bundle
$\overline{\M}$ on $\X$ such that $\overline{\M}_{a,b}^2>0$ implies that
$H^0\left(\X,\M_{a,b}^{\otimes n}\right)$ has a basis consisting of strictly effective sections
for $n$ large enough.
By~\cite[Theorem~1.3]{ZhangPos}, $\overline{\M}_{a,b}$ is ample.
 Lemma~\ref{Mample} implies that there is a multiple $n(a,b)$ of $n$ and an
effective section $s$ of
$\overline{\M}_{a,b}^{\otimes n(a,b)}$ such that ${\rm{div}(s)^{hor}}$ does not intersect
the finite set $T(\X)\subset \X$. 
Using Theorem~\ref{HtNonNeg} we conclude $h_{\overline{\L_D}}({\rm{div}(s)^{hor}})\ge0$.
\end{proof}
Now we can complete the proof of Proposition~\ref{LDBound}.
\begin{proof}[Proof of Proposition~\ref{LDBound}]
Suppose that $\overline{\L_D}^2<0$.
Let $\overline{\M}$ be as in Proposition~\ref{Mexists} and let $a,b$ be positive integers
such that $\overline{\M}_{a,b}^2>0$.
It follows from Lemma~\ref{Mlist} and Proposition~\ref{Mexists} that we have
\[\left(\overline{\L_D}\ldot \overline{\M}_{a,b}^{\otimes n(a,b)}\right)\ge0\] and thus 
\[(\overline{\L_D}\ldot \overline{\M}_{a,b})\ge 0.\]  
But by Lemma~\ref{ImitateZhang} this leads to a contradiction.
\end{proof}
Next we prove Theorem~\ref{thm-main}.
It follows from~\cite[Lemma~A.1]{CRCapacity} that there exists a $\Q$-divisor
$D\in\Div_\Q(X)$ such  that $D_\X\cap T(\X)=\emptyset$ and such that $(2g-2) D$ is a
canonical $\Q$-divisor on $X$.
Moreover, it is shown in~\cite{CurillaKuehn} that 
\begin{equation}\label{KXV}
        \K=(2g-2) (D_\X+V_D)\in\Div_\Q(\X)
\end{equation}
is a canonical $\Q$-divisor on $\X$. 
By the latter we mean a $\Q$-divisor such that $\O(\K)=\omega$.  
\begin{proof}[Proof of Theorem~\ref{thm-main}]
        From~\eqref{KXV} we get $D_\X=\frac1{2g-2} \K-V_D$.
        We can use this to rewrite $\overline{ \L_D}$ (cf. Definition~\ref{LDDef}) as
\[
        \overline{\L_D}=\om^{\otimes\frac{g}{g-1}}\otimes\O(-2V_D-U_D)-4g a,
\]
where
\[
        a=\Oa(D_\X)^2-\O(V_D)^2=\frac{1}{4(g-1)^2}\om^2-\frac{1}{g-1}(\om\ldot \O(V_D)).
\]
Hence we have
\[
\overline{ \L_D}^2=\frac{1}{g-1}\left(g\om^2+(g-1)\O(2V_D+U_D)^2-2g(\om\ldot
\O(U_D))\right).
\]
Since $\overline{\L_D}^2\ge0$ by Proposition~\ref{LDBound}, Theorem~\ref{thm-main} follows.
\end{proof}
\begin{rk}
  We have developed the theory of $\overline{\L_D}$ for rather general degree
  one divisors $D\in\Div_\Q(X)$.
  The main reason why we choose to work with $D$ as in Theorem~\ref{thm-main} is 
  that $D_\X$ has an obvious relation with   $\om$.
  But there are other promising choices for $D$; for instance, we could take
  $D=\frac{1}{g^3-g}W$, where $W$ is the divisor of Weierstrass points on $X$.
  This was suggested by Ariyan Javanpeykar.
  In fact, it is easy to see that the divisor $\mathcal{V}$ used in~\cite[Lemma~5.1]{deJong}
  to extend $W$ to a divisor on $\X$ with good properties is a valid choice for
  $V_W$.
\end{rk}
Now we derive Theorem~\ref{thm-main} from Theorem~\ref{thm-simple}.
\begin{proof}[Proof of Theorem~\ref{thm-simple}]
Let $D \in \Div_\Q(X)$ have degree one.
    We know that $\beta=\beta_D$ does not depend on the choice of $D$ because of
    Lemma~\ref{betaformula}.
    In order to apply Theorem~\ref{thm-main}, we first need to show that
    $\overline{\L_D}$ is relatively semipositive.
Let $\Gamma$ be an irreducible component of a special fiber of $\X$.
By definition of $\overline{\L_D}$, we have
\[	(\overline{\L_D}\ldot \O(\Gamma))=
      (\K\ldot \Gamma)+2(D_\X\ldot \Gamma)-(U_{D}\ldot \Gamma)
  \]
  up to a positive rational constant, where $\K$ is a canonical divisor on $\X$.
  Hence relative semipositivity of $\overline{\mathcal{L}_D}$ follows from part (b) of Lemma~\ref{HDGi}.
Since we know that the assumptions of Theorem~\ref{thm-main} are satisfied,
it suffices to show that $(\K\ldot U_D)\ge 0$.
But this is an immediate consequence of Lemma~\ref{udk}.
\end{proof}
\begin{rk}\label{bogomolov}
One of the main original motivations to look at lower bounds for $\om^2$ was a conjecture
of Bogomolov.
Building on earlier work of Zhang~\cite{ZhangAdmissible},
the conjecture was finally proved by Ullmo~\cite{Ullmo}, who proved the positivity of the
admissible self-intersection $\omega_a^2$ of $\omega$ for curves over number fields.
For curves over function fields of characteristic~0, the conjecture was reduced by
Zhang~\cite{ZhangGrossSchoen} to a
conjecture about invariants of polarized metrized graphed and the latter was proved by
Cinkir~\cite{CinkirBogomolov}.
However, Cinkir actually proved an {\em effective} version of the Bogomolov conjecture.
Such an effective version can also be conjectured over number fields, but in this
situation it has not
been proved yet (it would follow from a proof of the arithmetic standard conjectures of
Gillet-Soul\'e, see~\cite[\S1.4]{ZhangGrossSchoen}).
Using \cite{ZhangAdmissible}, it suffices to find an effectively computable nontrivial lower 
bound for $\omega^2_a$.
If $\X$ is semistable and minimal, then we have 
\[
\om^2=\omega^2_a-\sum_v \eps_v(X),
\]
where $\eps_v(X)\ge0$ is Zhang's admissible constant associated to $X\times K_v$ (see
Remark~\ref{constants}) and the sum is over all non-archimedean places of $K$.
Hence it suffices to find an effectively computable lower 
bound $b$ for $\om^2$ such that 
$\sum_v \eps_v(X)<b$.
Therefore our work provides a possible approach to the effective Bogomolov Conjecture,
but unfortunately we already have $\beta \le \sum_v \eps_v$ for $g=2$
by Remark~\ref{eps-beta-g2}. 
\end{rk}
\section{Applications}\label{sec:apps}
Now we apply our results to compute lower bounds on the self-intersection of the 
relative dualizing sheaf for certain families of curves. 
\subsection{Modular curves}\label{x1n}
Let $N =N'QR$ be a squarefree integer such that $Q,R\ge 4$ and $\gcd(Q,R)=1$.
Consider the modular curve $X_1(N)$ over the cyclotomic field $\Q[\zeta_N]$ and its
minimal regular model $\X = \X_1(N)/\Z[\zeta_N]$.
Then $\X$ has semistable and reduced fibers. More precisely, the special fibers $\X_\fp$
are smooth if $\fp \nmid N$. If $\fp \mid N$ with residue characteristic $p$, then the
special fiber $\X_\fp$ consists of two isomorphic 
curves intersecting in 
\[
    s_\fp = \frac{p-1}{24}\frac{\varphi(N/p)N}{p}\prod_{q\mid N/p} (1+\frac{1}{q})
\]
points, all of which are rational over the residue field at $\fp$,
see~\cite[Proposition~7.3]{Mayer},
The arithmetic genera of these components are given by
\[
q_\fp = \frac{1}{2}(g_N-s_\fp+1),
\]
where 
\[
g_N = 1 +
\frac{1}{24}\varphi(N)N\prod_{p|N}(1+\frac{1}{p})-\frac{1}{4}\sum_{d|N}\varphi(d)\varphi(N/d)
\]
 is the genus of $X_1(N)$.
We can use Example~\ref{betaex} to compute an asymptotic lower bound for $\om^2$ quite easily:
\begin{prop}
The arithmetic self-intersection $\om^2$ of the relative dualizing sheaf on
$\X(N)$ satisfies
  \[
    \om^2\;\ge\;\frac{1}{2}\varphi(N)\log N + {o}(1),
  \]
\end{prop}
\begin{proof}
Let $n_\fp =\log\#k(\fp)$.
Then we have $\sum_{\fp|p} n_\fp = \varphi(N/p)\log(p)$ and hence, by
Example~\ref{betaex},
\begin{align*}
\beta&=\sum_{\fp|N}\frac{n_\fp}{2s_{\fp}}(s_\fp+2q_\fp-2)\\
&=\sum_{\fp|N}\frac{n_\fp}{2s_{\fp}}(g_N-1)\\
    &= \frac{g_N-1}{2}\sum_{p|N}\sum_{\fp|p}\frac{n_\fp}{s_{\fp}}\\
    &= 12(g_N-1)\frac{\varphi(N)}{\prod_{p|N}p^2-1}\sum_{p|N}\frac{p+1}{p-1}\log p\\
    &=\frac{1}{2}\varphi(N)\log N + {o}(1),
\end{align*}
since $\frac{24(g_N-1)}{\prod_{p|N}p^2-1}=1 + {o}(1)$.
\end{proof}
\begin{rk}
In~\cite[Theorem~7.7]{Mayer}, Mayer obtains the asymptotic formula
\[
    \om^2\;=\;3g_N\log(N) + {o}(g_N\log(N)).
\]
Our lower bound is much smaller than this asymptotic value.
\end{rk}
\subsection{Fermat curves of prime exponent}\label{Fermat}
In Section~\ref{Bds} we derived a nontrivial lower bound $\beta_D$ on $\om^2$ for
minimal arithmetic surfaces with simple multiplicities.
In the present subsection we compute lower bounds on $\om^2$ in a situation where Theorem
\ref{thm-simple} is not applicable, namely for minimal regular
models of Fermat curves of prime exponent.
Along the way, we construct $U_D$ and show that $\overline{\L_D}$ is
relatively semipositive, where $D=(S_x)$ for a certain rational point $S_x$ such that $(2g-2)D$ is a
canonical $\Q$-divisor on $\X$.
We start with a brief review of the notation from~\cite{CurillaKuehn}.
Let $p>3$ be a prime number, let \[
F_p:X^p+Y^p=Z^p
\] denote the Fermat curve with exponent $p$ over
$K=\Q(\zeta_p)$, where $\zeta_p$ is a primitive $p$-th root of unity. Let
$\sF^{\min}_p$ denote the minimal regular model of $F_p$ over $\Z[\zeta_p]$ as computed by McCallum cf.~\cite{McCallum}. 
We denote the components of the only non-reduced special fiber of $\sF^{\min}_p$ by
$L_i,i\in I$, where 
\[
I=\{x,y,z,\;\alpha_1,\;\alpha_{1,1},\ldots,\alpha_{1,p},\;\alpha_2,\;\alpha_{2,1},\ldots,\alpha_{2,p},\;\ldots,\;\alpha_r,\;\alpha_{r,1},\ldots,\alpha_{r,p},\;\beta_1,\ldots,\beta_s\}\]
and $2r+s=p-3$, see~\cite{CurillaKuehn}.
In Figure~\ref{d1}  the configuration of the
only reducible special fiber (occuring at the unique prime above $p$) of a certain non-minimal
model $\sF_p$ of $F_p$ is shown. 
It has the property that contracting the unique exceptional component $L$ on
$\sF_p$ yields $\sF^{\min}_p$. 
For every component $L_i$, we also list the pair $(m_i, L_i^2)$, where $m_i$ is the
multiplicity of $L_i$ in $\sF_p$. 
All components have genus 0 and the only component with self-intersection number~-1 is $L$.
See~\cite{McCallum,CurillaKuehn} for further details.
\begin{figure}[ht]\centering
 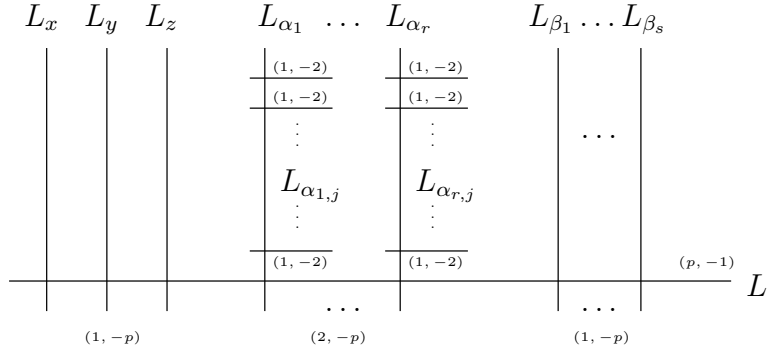
  \parbox{12cm}{\caption{The configuration of the only reducible special fiber of
  $\sF_{p}$. }\label{d1}}
  \end{figure}
Note that $L_{\alpha_i}$ has multiplicity two for all $i$.
Therefore 
Theorem~\ref{thm-simple} does not
apply in general and we have to show relative semipositivity of $\overline{\L_D}$ and nontriviality of the
bound from Theorem~\ref{thm-main} directly for a suitable $\Q$-divisor $D$. 
Since there is only one place of bad reduction, we will omit it from the
notation for the sake of simplicity.
It is shown by Curilla and the first author in~\cite{CurillaKuehn} that  a canonical divisor on
$F_p$ is given by $(2g-2) D$, where $D=S_x$ for a certain
$K$-rational point $S_x$ on $F_p$ whose Zariski closure $\mathcal{S}_x := S_{x,
\sF^{\min}_p}$ in $\sF^{\min}_p$ intersects only $L_x$.
In order to compute a lower bound on $\om^2$ using the results of the previous
sections, we will find $U_D$.
This means that we first have to compute the divisor $V_i$ for each $i\in I$.
\begin{lemma}\label{VGis}
 We can take  \begin{itemize}
   \item $V_i=\frac{1}{p}L_i$, if $i\in\{x,y,z,\beta_j\}$,
      \item
        $V_{\alpha_i}=\frac{1}{p}L_{\alpha_i}+\frac{1}{2p}\sum^p_{j=1}L_{\alpha_{ij}}$,
      \item 
        $V_{\alpha_{i,j}}=-\frac{1}{p}L_{\alpha_i}+\left(\frac{1}{2}-\frac{1}{2p}\right)L_{\alpha_{ij}}-\frac{1}{2p}\sum_{k\ne
        j}L_{\alpha_{ik}}$.
    \end{itemize}
\end{lemma}
\begin{proof}
  Recall that the divisor $V_i$ must satisfy
  \begin{equation}\label{VGi}
  (V_i\ldot L_j)=a'_j-\delta_{ij}
    \end{equation}
  for all $j\in I$, where $a'_j=\frac{1}{2g-2}(-L_j^2+2p_a(L_j)-2)$ and
  $\delta_{ij}$ is the Kronecker delta function on $I$.
  Since all components have genus zero and self-intersection $1-p$ except for
  the components $L_{\alpha_{ij}}$, which have genus zero and self-intersection
  $-2$, we get $a'_j=0$ for the components $L_{\alpha_{ij}}$ and $a'_j=\frac{1}{p}$ for all other components.
  Checking the validity of the Lemma reduces to checking~\eqref{VGi} for each
  $i\in I$ which is a simple computation that we leave to the reader.
  Note that for $V_x$ this was essentially shown in~\cite[Proposition~8.3]{CurillaKuehn}.
  \end{proof}
\begin{cor}\label{HK}
  A canonical $\Q$-divisor on $\sF^{\min}_p$ is given by
  $(2g-2)\mathcal{S}_x+\frac{1}{p}L_x$ and we have
  \begin{align*}
    U_D =&
    \frac{1-p}{p^2}L_x+\frac{1+p}{p^2}\left(L_y+L_z+\sum^s_{j=1}L_{\beta_j}\right)+\frac{1+p/2}{p^2}\sum^r_{i=1}L_{\alpha_i}\\&+\frac{p^2/2+p/2-3}{p^2}\sum^r_{i=1}\sum^p_{j=1}L_{\alpha_{ij}}.
  \end{align*}
\end{cor}
\begin{proof}
  The first assertion follows immediately from~\eqref{KXV}, the fact that $(2g-2)S_x$ is a
  canonical divisor on $F_p$ and Lemma~\ref{VGis}. 
  For the second assertion, recall the definition of $U_D$: 
 
  \[
    U_D=\sum_{i\in I}\left(2(V_i\ldot V_x)-V_i^2\right) L_i
  \]
  Next we compute
  \begin{align*}
    V_i^2\;=\;&\left(\frac{1}{p}L_i\right)^2=\frac{1}{p^2}(1-p)\textrm{
  for }i\in\{x,y,z,\beta_j\}\\
  (V_i\ldot V_x)\;=\;&\left(\frac{1}{p}L_i\ldot \frac{1}{p}L_x\right)=\frac{1}{p^2}\textrm{
  for }i\in\{y,z,\beta_j\}\\
  (V_{\alpha_i}\ldot
  V_x)\;=\;&\left(\frac{1}{p}L_{\alpha_i}+\frac{1}{2p}L_{\alpha_{ij}}\ldot \frac{1}{p}L_x\right)=\frac{1}{p^2}\\
(V_{\alpha_{ij}}\ldot V_x)\;=\;&\left(-\frac{1}{p}L_{\alpha_i}+\left(\frac{1}{2}-\frac{1}{2p}\right)L_{\alpha_{ij}}-\frac{1}{2p}\sum_{k\ne
j}L_{\alpha_{ik}}\ldot \frac{1}{p}L_x\right)=-\frac{1}{p^2}\\
  \end{align*}
  \begin{align*}
(V_{\alpha_i})^2\;=\;&\left(\frac{1}{p}L_{\alpha_i}+\frac{1}{2p}L_{\alpha_{ij}}\right)^2=\frac{1}{p^2}(1-p)+\frac{p}{p^2}-\frac{2p}{4p^2}
=\frac{1}{p^2}-\frac{1}{2p}\\
(V_{\alpha_{ij}})^2\
  \;=\;&\left(-\frac{1}{p}L_{\alpha_i}+\left(\frac{1}{2}-\frac{1}{2p}\right)L_{\alpha_{ij}}-\frac{1}{2p}\sum_{k\ne
  j}L_{\alpha_{ik}}\right)^2\\
  \;=\;&\frac{1}{p^2}(1-p)-\frac{1}{2}\left(1-\frac{1}{p}\right)^2-\frac{1}{2p^2}(p-1)-\frac{1}{p}\left(1-\frac{1}{p}\right)+\frac{p-1}{p^2}\\
  \;=\;&\frac{1}{2p^2}(2-p-p^2)
  \end{align*}
  A simple computation proves the corollary.
\end{proof}
\begin{lemma}\label{FpRelSem}
  The hermitian line bundle $\overline{\L_D}$ is relatively semipositive.
\end{lemma}
\begin{proof}
  Recalling Definition~\ref{LDDef} of $\overline{\L_D}$ we see that we need to
  show
  \begin{equation}\label{LKinequ}
    a_i+2(\mathcal{S}_x\ldot L_i)-(U_D\ldot L_i)\ge 0
  \end{equation}
  for all $i\in I$, where $a_i=0$ for $i=\alpha_{ij}$ and $a_i=p-3$ for all
  other components.
  As usual, we will distinguish between the different components $L_i$ to prove
 ~\eqref{LKinequ}.
  Throughout, we will use that $s=p-3-2r$ to eliminate $s$.
  We start with $L_x$ and find, using Corollary~\ref{HK}, that
  \[
    (U_D\ldot L_x)=\frac{1-p}{p^2}(1-p)+(s+2)\frac{1+p}{p^2}+r\frac{1+p/2}{p^2}=2-\frac{2}{p}-\frac{r}{p^2}-\frac{3r}{2p}
  \]
  which gives
  \[
    a_x+2(\mathcal{S}_x\ldot L_x)-(U_D\ldot L_x)\ge p-3>0.
  \]
  Let $i\in\{y,z,\beta_j\}$. 
  Then we get
  \[
    (U_D\ldot L_i)=\frac{1-p}{p^2}+\frac{1+p}{p^2}(1-p)+(s+1)\frac{1+p}{p^2}+r\frac{1+p/2}{p^2}=-\frac{2}{p}-\frac{r}{p^2}-\frac{3r}{2p}
  \]
  and thus
  \[
    a_i+2(\mathcal{S}_x\ldot L_i)-(U_D\ldot L_i)\ge p-3>0.
  \]
  It remains to consider the components $L_{\alpha_i}$ and $L_{\alpha_{ij}}$.
  The computations are similar to the ones above, but tedious and hence are
  omitted.
  The upshot is that we get inequalities
  \[
    a_{\alpha_i}+2(\mathcal{S}_x\ldot L_{\alpha_i})-(U_D\ldot L_{\alpha_i})=
    p-3-\frac{1}{2}p+1+\frac{5}{p}+\frac{r}{p^2}+\frac{3r}{2p}\ge
    \frac{p}{2}-2>0
  \]
  and
  \[
    a_{\alpha_{ij}}+2(\mathcal{S}_x\ldot L_{\alpha_{ij}})-(U_D\ldot L_{\alpha_{ij}})=
    1+\frac{1}{2p}-\frac{7}{p^2}>0.    
  \]
  This shows that~\eqref{LKinequ} is satisfied for all components $L_i$, which
  proves the lemma.
\end{proof}
\begin{thm}\label{FpBound} Let $p>3$ be a prime number and let $\om^2$ denote the
    arithmetic self-intersection of the relative dualizing sheaf on $\sF^{\min}_p$.
\\(i) We have
  \begin{align*}
    \om^2\;\ge\;&
    \frac{1}{4p^3(p-1)(p-2)}\left((4+2r)p^6-(32+10r)p^5+(10+19r)p^4+(124\right.\\&-r-25r^2)p^3+(-56+52r+31r^2)p^2+(156-328r+112r^2)p\\&\left.+144-24r+60r^2\right)\log p
  \end{align*}
  (ii) We have 
  \begin{align*}
    \om^2\;\ge\;&\frac{\left(4p^6-32p^5+\frac{13}{2}p^4+\frac{73}{2}p^3-52p^2+144\right)
}{4p^3(p-1)(p-2)}\log p.
  \end{align*}
This lower bound is positive for all $p>7$.
Furthermore, if $p=5$, then we have $\om^2\ge\frac{188}{125}\log 5$ and if
$p=7$, then we have
$\om^2\ge\frac{37277}{6860}\log 7$.
\end{thm}
\begin{proof}
  By Theorem~\ref{thm-main} and Lemma~\ref{FpRelSem} we know that
  \[
    \om^2\ge \beta_D =-\frac{g-1}{g}\O(2V_D+U_D)^2+2(\om\ldot \O(U_D)).
  \]
  We compute the terms on the right hand side.
  
  First note that
  \begin{align*}
    2V_D+U_D=&\;
    \frac{1+p}{p^2}\left(L_x+L_y+L_z+\sum^s_{j=1}L_{\beta_j}\right)\\&\;+\frac{1+p/2}{p^2}\sum^r_{i=1}L_{\alpha_i}+\frac{p^2/2+p/2-3}{p^2}\sum^r_{i=1}\sum^p_{j=1}L_{\alpha_{ij}}
\end{align*}
Using this, it is not hard to verify that
\begin{align}\label{VKHKSq}
(2V_D+U_D)^2=&\;
-\frac{pr}{2}-\frac{r}{2}+1+\frac{1}{p}(\frac{7}{4}r-5)+\frac{1}{p^2}(\frac{25}{4}r^2-5r-1)\nonumber\\
&\; +\frac{1}{p^3}(17-30r+11r^2)+\frac{1}{p^4}(12-2r+5r^2).
\end{align}
For the computation of $(\mathcal{K}\ldot U_D)$, where $\mathcal{K}$ is a
canonical divisor on $\sF^{\min}$, we use the adjunction formula.
Namely, if $\theta_i$ denotes the multiplicity of $L_i$ in $U_D$, then we have
\[
(\mathcal{K}\ldot U_D)=\sum_{i\in I}\theta_ia_i
\]
and hence
\begin{equation}\label{AdjHK}
(\mathcal{K}\ldot U_D)=(p-3)\left(\frac{1-p}{p^2}+(s+2)\frac{1+p}{p^2}+r\frac{1+p/2}{p^2}\right).
\end{equation}
A combination of~\eqref{VKHKSq} and~\eqref{AdjHK} proves (i) after a
little algebra.
For (ii) we use (i) and $0\le r \le \frac12p-\frac32$.
To derive the lower bounds for $p=5,7$, we use that $r=0$ if $p=5$ and $r=2$ if
$p=7$.
\end{proof}
The proof of Theorem~\ref{thm-FpBound2} follows immediately from Theorem~\ref{FpBound}.
\begin{rk}
The upper bound computed by Curilla and the first author in~\cite{CurillaKuehn}
is of order $\O(gp\log p)$, i.e. of order $\O(p^3\log p)$. 
\end{rk}

\end{document}